\documentclass[11pt]{amsart}
\usepackage[top=1.5in, bottom=1.5in, left=1.4in, right=1.4in]{geometry} 
\usepackage{color}
\usepackage{graphicx}
\usepackage{amssymb}
\usepackage{epstopdf}
\usepackage{amsthm}
\usepackage{hyperref}
\usepackage[table]{xcolor}

\usepackage{array}
\usepackage{mathtools}
\usepackage{enumerate}
\usepackage{enumitem}
\usepackage{lscape}
\usepackage{verbatim}
\usepackage{color}

\allowdisplaybreaks

\usepackage[all]{xy}
\usepackage{graphicx}  
\usepackage{tikz}
\usetikzlibrary{decorations.pathreplacing}
\usepackage{tkz-graph}
\usetikzlibrary{arrows}
\usetikzlibrary{decorations.markings}
\usepackage{xcolor}
\usepackage{ytableau}
\ytableausetup{boxsize=1.5em}

\newcommand{\B}{\mathcal{B}}
\newcommand{\J}{\mathcal{J}}
\newcommand{\mcO}{I}
\newcommand{\T}{\mathcal{T}}
\newcommand{\jag}{\mathrm{jag}}
\newcommand{\Poly}{\mathcal{P}}

\newcommand{\Orb}{\mathcal{O}}
\newcommand{\ra}{\rightarrow}
\newcommand{\NN}{\mathbb{N}}
\newcommand{\ZZ}{\mathbb{Z}}
\newcommand{\RR}{\mathbb{R}}
\newcommand{\EE}{\mathbb{E}}
\newcommand{\PP}{\mathbb{P}}
\newcommand{\Ind}{\mathbf{1}}
\newcommand{\lin}{\mathrm{lin}}
\newcommand{\unif}{\mathrm{unif}}
\newcommand{\rk}{\mathrm{rk}}

\newcommand{\col}{\colon}
\newcommand{\coleq}{\coloneqq}

\newtheorem{thm}{Theorem}[section]
\newtheorem{lemma}[thm]{Lemma}

\newtheorem{cor}[thm]{Corollary}
\newtheorem{prop}[thm]{Proposition}

\newtheorem{Definition}[thm]{Definition}
\newenvironment{definition}
  {\begin{Definition}\rm}{\end{Definition}}

\newtheorem{Example}[thm]{Example}
\newenvironment{example}
  {\begin{Example}\rm}{\end{Example}}
  
\newtheorem{Remark}[thm]{Remark}
\newenvironment{remark}
  {\begin{Remark}\rm}{\end{Remark}}
  
\numberwithin{equation}{section}

\usepackage{etoolbox}
\apptocmd{\sloppy}{\hbadness 10000\relax}{}{}

\title{The expected jaggedness of order ideals}
\author[M. Chan]{Melody Chan}\address{Department of Mathematics, Harvard University, Cambridge, MA 02138}\email{mtchan@math.harvard.edu}
\author[S. Haddadan]{Shahrzad Haddadan}\address{6211 Sudikoff Lab, Dartmouth College, Hanover, NH 03755}\email{shahrzad@cs.dartmouth.edu}
\author[S. Hopkins]{Sam Hopkins}\address{Department of Mathematics, MIT, Cambridge, MA 02139}\email{shopkins@mit.edu}
\author[L. Moci]{Luca Moci}\address{IMJ-PRG, Universit\'{e} Paris-Diderot Paris 7, Paris, France}\email{lucamoci@hotmail.com}

\begin{document}

\begin{abstract}
The {\em jaggedness} of an order ideal $I$ in a poset $P$ is the number of maximal elements in $I$ plus the number of minimal elements of $P$ not in $I$.  A probability distribution on the set of order ideals of $P$ is {\em toggle-symmetric} if for every $p\in P$, the probability that $p$ is maximal in $I$ equals the probability that $p$ is minimal not in $I$.  In this paper, we prove a formula for the expected jaggedness of an order ideal of $P$ under any toggle-symmetric probability distribution when $P$ is the poset of boxes in a skew Young diagram.  Our result extends the main combinatorial theorem of Chan-L\'opez-Pflueger-Teixidor \cite{chan2015genera}, who used an expected jaggedness computation as a key ingredient to prove an algebro-geometric formula; and it has applications to homomesies, in the sense of Propp-Roby, of the antichain cardinality statistic for order ideals in partially ordered sets.
\end{abstract}

\maketitle

\section{Introduction}\label{s:intro}

Consider an $a\times b$ grid, and consider the set of~$\binom{a+b}{a} $ lattice paths from the lower-left corner of this grid to the upper-right corner.  Define a probability distribution on this set as follows: a path occurs with probability proportional to the number of $a\times b$ standard Young tableaux with which it is compatible. (We say a path $s$ is compatible with a tableau~$T$ if all the labels of $T$ northwest of $s$ are smaller than all of the labels of~$T$ southeast of~$s$.  We use English notation throughout.)  We will call this distribution~$\mu_{\lin}$, the {\em linear} distribution on lattice paths, since it comes from linear orderings of the $ab$ boxes in the grid.
For example, the six lattice paths in a $2\times 2$ grid occur in $\mu_{\lin}$ with the probabilities shown below:

\medskip

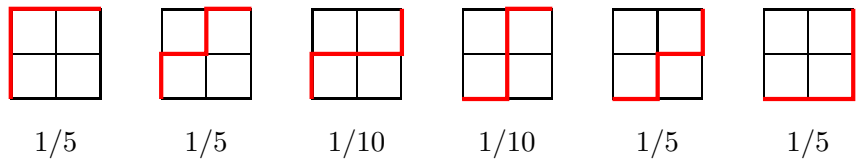
\begin{center}
\begin{tikzpicture}
\begin{scope}[shift={(-5,0)},scale=1.0]
\node at (0,0) {\ydiagram{2,2}};
\node at (0,-1.2) {$1/5$};
\def\x{0.6}
\draw[red,ultra thick] (-1*\x,-1*\x) -- (-1*\x,1*\x) -- (1*\x,1*\x);
\end{scope}
\begin{scope}[shift={(-3,0)},scale=1.0]
\node at (0,0) {\ydiagram{2,2}};
\node at (0,-1.2) {$1/5$};
\def\x{0.6}
\draw[red,ultra thick] (-1*\x,-1*\x) -- (-1*\x,0) -- (0,0) -- (0,1*\x) -- (1*\x,1*\x);
\end{scope}
\begin{scope}[shift={(-1,0)},scale=1.0]
\node at (0,0) {\ydiagram{2,2}};
\node at (0,-1.2) {$1/10$};
\def\x{0.6}
\draw[red,ultra thick] (-1*\x,-1*\x) -- (-1*\x,0) -- (1*\x,0) -- (1*\x,1*\x);
\end{scope}
\begin{scope}[shift={(1,0)},scale=1.0]
\node at (0,0) {\ydiagram{2,2}};
\node at (0,-1.2) {$1/10$};
\def\x{0.6}
\draw[red,ultra thick] (-1*\x,-1*\x) -- (0,-1*\x) -- (0,1*\x) -- (1*\x,1*\x);
\end{scope}
\begin{scope}[shift={(3,0)},scale=1.0]
\node at (0,0) {\ydiagram{2,2}};
\node at (0,-1.2) {$1/5$};
\def\x{0.6}
\draw[red,ultra thick] (1*\x,1*\x) -- (1*\x,0) -- (0,0) -- (0,-1*\x) -- (-1*\x,-1*\x);
\end{scope}
\begin{scope}[shift={(5,0)},scale=1.0]
\node at (0,0) {\ydiagram{2,2}};
\node at (0,-1.2) {$1/5$};
\def\x{0.6}
\draw[red,ultra thick] (-1*\x,-1*\x) -- (1*\x,-1*\x) -- (1*\x,1*\x);
\end{scope}
\end{tikzpicture}
\end{center}

\noindent We may ask: what is the expected {\em jaggedness} of a lattice path, chosen according to~$\mu_{\lin}$?  That is, what is the expected number of turns of such a lattice path?  The answer is surprisingly simple.  

\begin{thm}\label{t:clpt}
The expected jaggedness of a lattice path in an $a\times b$ grid, chosen under the distribution~$\mu_{\lin}$, is exactly $2ab/(a\!+\!b)$, the harmonic mean of $a$ and $b$.  
\end{thm}

Theorem~\ref{t:clpt} and a generalization thereof appeared recently in \cite{chan2015genera} as the key combinatorial result underlying the computation of the genera of {\em Brill-Noether curves} (Brill-Noether loci of dimension 1).  Briefly: it is used to compute the average vertex degree in the dual graph of a nodal degeneration, parametrizing Eisenbud-Harris limit linear series, of a given Brill-Noether curve.  Noted in \cite{chan2015genera} is the unexpected appearance of the harmonic mean, as well as the observation that if the distribution $\mu_{\lin}$ is replaced by the {\em uniform} distribution, the answer is still the harmonic mean.

The purpose of this paper is to give a vast generalization of Theorem~\ref{t:clpt}, in particular explaining the seeming coincidence above, and putting it in its proper combinatorial context: order ideals in arbitrary posets, and {\em toggle-symmetric} probability distributions on them.  This last is a class of probability distributions that we would like to put forth as an interesting property to study, especially in relation to the developing area of {dynamical algebraic combinatorics}.  We define toggle-symmetric distributions on order ideals of posets, and give, with proof, many natural examples, in Section~\ref{sec:togglesym}.  The word ``toggle'' refers to the procedure of adding or removing an element from a set if it is permissible to do so.  The term was coined by Striker-Williams \cite{striker2015toggle} in describing Cameron and Fon-Der-Flaass' involutions on sets of order ideals of posets \cite{cameron1995orbits}. Indeed, our results have direct applications to homomesy results for order ideals under special compositions of toggles, as we will discuss.

In Section~\ref{sec:exp}, we prove our main result: a formula for expected jaggedness that applies to all skew Young diagrams, not just rectangles, and any toggle-symmetric distribution.  Here is our main theorem:
\begin{thm}\label{thm:main_intro}
Let $\sigma$ be a connected skew shape with height $a$ and width $b$.  Let $\mu$ be any toggle-symmetric probability distribution on the subshapes of $\sigma$.  Then the expected jaggedness of a subshape of $\sigma$ with respect to $\mu$ is
\begin{equation}\label{e:main_intro}
\frac{2ab}{a+b}\left( 1 + \sum_{c\in C(\sigma)} \!\delta(c)\,\PP_\mu(c)\right).
\end{equation}
\end{thm}
\noindent Here:
\begin {itemize}
\item $C(\sigma)$ is the set of {\em outward corners} of $\sigma$, and $\PP_{\mu}(c)$ is the probability (according to~$\mu$) that 
the edges of the outward corner $c$ are both included in the lattice path that ``cuts out'' the subshape 
(see Definition \ref{def:outward} for details); 
\item the \emph{displacement} $\delta(c)$ is proportional to the signed distance between the corner~$c$ and the antidiagonal of the partition (Definition \ref{def:disp}).
\end{itemize}
For now, the main point is that the expected jaggedness can be calculated as the harmonic mean of~$a$ and~$b$, plus a sum of correction terms that can be completely understood in terms of~$\mu$ and~$\sigma$.   (When $\sigma$ is a rectangle, there are no correction terms and~\eqref{e:main_intro} gives the harmonic mean exactly, for {\em any} toggle-symmetric distribution.)  

There are several key differences between Theorem~\ref{thm:main_intro} and the corresponding result~\cite[Theorem 2.8]{chan2015genera} of Chan et al.  First, our theorem applies to {\em any} toggle-symmetric distribution.  Moreover, it is fully symmetric with respect to interchanging rows and columns, which is not the case in \cite{chan2015genera}.  Indeed, our result makes explicit that the only dependence is on the outer corners and their displacements. This will allow us to immediately derive that for any {\em balanced} shape, the expected jaggedness is always the harmonic mean; see Corollary~\ref{cor:main}.  

We also note that our results, combined with theorems of Striker~\cite{striker2015toggle}, have direct applications to homomesy results under the operations of rowmotion and gyration on posets. In particular, they allow us to recover and generalize a theorem of Propp and Roby \cite{propp2013homomesy} on homomesies for antichain cardinalities.  We explain these applications in~\S\ref{ssec:antichain}. In Section~\ref{sec:open}, we give four open questions we would like to see explored.

We close this section by giving an example that illustrates Theorem~\ref{thm:main_intro}.

\begin{example}
Consider the Young diagram shape $\sigma=(3,1)$. The seven subshapes of $\sigma$, equivalently the lattice paths in $\sigma$, are depicted below. The numbers below each path indicate that subshape's jaggedness, along with the probability of that subshape's occurrence according to the linear distribution.  Then we can calculate directly that~$\EE_{\mu_\lin}(\jag)=34/15 $.

\medskip

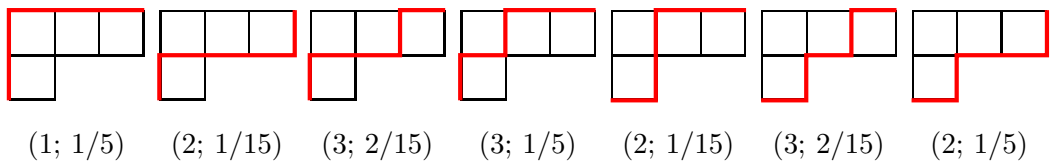
\begin{center}
\begin{tikzpicture}
\begin{scope}[shift={(-7,0)},scale=1.0]
\node at (0,0) {\ydiagram{3,1}};
\node at (0,-1.2) {$(1;\, 1/5)$};
\draw[red,ultra thick] (-0.9,-0.6) -- (-0.9,0.6) -- (0.9,0.6);
\end{scope}

\begin{scope}[shift={(-5,0)},scale=1.0]
\node at (0,0) {\ydiagram{3,1}};
\node at (0,-1.2) {$(2;\, 1/15)$};
\draw[red,ultra thick] (-0.9,-0.6) -- (-0.9,0) -- (0.9,0)-- (0.9,0.6);
\end{scope}

\begin{scope}[shift={(-3,0)},scale=1.0]
\node at (0,0) {\ydiagram{3,1}};
\node at (0,-1.2) {$(3;\, 2/15)$};
\draw[red,ultra thick] (-0.9,-0.6) -- (-0.9,0) -- (0.3,0)-- (0.3,0.6)--(0.9,0.6);
\end{scope}

\begin{scope}[shift={(-1,0)},scale=1.0]
\node at (0,0) {\ydiagram{3,1}};
\node at (0,-1.2) {$(3;\, 1/5)$};
\draw[red,ultra thick] (-0.9,-0.6) -- (-0.9,0) -- (-0.3,0)-- (-0.3,0.6)--(0.9,0.6);
\end{scope}

\begin{scope}[shift={(1,0)},scale=1.0]
\node at (0,0) {\ydiagram{3,1}};
\node at (0,-1.2) {$(2;\, 1/15)$};
\draw[red,ultra thick] (-0.9,-0.6) -- (-0.3,-0.6) -- (-0.3,0)-- (-0.3,0.6)--(0.9,0.6);
\end{scope}

\begin{scope}[shift={(3,0)},scale=1.0]
\node at (0,0) {\ydiagram{3,1}};
\node at (0,-1.2) {$(3;\, 2/15)$};
\draw[red,ultra thick] (-0.9,-0.6) -- (-0.3,-0.6) -- (-0.3,0) -- (0.3,0)-- (0.3,0.6)--(0.9,0.6);
\end{scope}

\begin{scope}[shift={(5,0)},scale=1.0]
\node at (0,0) {\ydiagram{3,1}};
\node at (0,-1.2) {$(2;\, 1/5)$};
\draw[red,ultra thick] (-0.9,-0.6) -- (-0.3,-0.6) -- (-0.3,0) -- (0.9,0)--(0.9,0.6);
\end{scope}
\end{tikzpicture}
\end{center}

\noindent Now, let us use Theorem \ref{thm:main_intro} instead to compute $\EE_{\mu_\lin}(\jag)$, using the fact that $\mu_\lin$ is toggle-symmetric by Proposition~\ref{prop:limit} and Remark~\ref{r:shuffles}. The corner $c$ occurring at~$(1,1)$ is the only outward corner of~$\sigma$.  Its displacement $\delta(c)$ is $-1/6$, as in Definition~\ref{def:disp}.  Finally, we have $\PP_{\mu_{\lin}}(c)=1/3$, by the formula~\eqref{eq:mulin} obtained in \S\ref{sec:correction}. Plugging these values into ~\eqref{e:main_intro} yields
\[\EE(\jag)=(12/5)\left( 1-1/6 \cdot 1/3\right) = 34/15,\] 
as expected.

\end{example}

\bigskip

\noindent {\bf Acknowledgments.} This paper grew out of discussions at the 2015 conference on Dynamical Algebraic Combinatorics at the American Institute of Mathematics (AIM).  We thank the organizers of the conference, J.~Propp, T.~Roby, J.~Striker, and N.~Williams, as well as all of the staff at AIM, for providing a very stimulating environment for collaboration.  We acknowledge SageMathCloud for providing a useful platform for remote collaboration. Finally, we thank R.~Stanley for pointing out the use of the Reciprocity Theorem for order polynomials to compute the number of increasing tableaux of bounded height, as detailed in Section~\ref{sec:correction}.  M.~Chan was supported by NSF DMS-1204278. S.~Hopkins was supported by NSF grant \#1122374.

\section{Toggle-symmetric distributions} \label{sec:togglesym}

For background on posets see~\cite[\S3]{stanley2012ec1}. Fix a finite poset $(P,\le)$.  An {\em order ideal} of~$P$ is a subset~$\mcO \subseteq P$ such that for every~$p\in \mcO$ and every $q \in P$ with $q \le p$, we have~$q\in \mcO$.  We denote the set of order ideals of $P$ by~$\J(P)$. If $P = P_1 \sqcup P_2$ then the set of order ideals decomposes as~$\J(P) = \J(P_1) \times \J(P_2)$ so we will assume from now on that~$P$ is connected. We do not consider the empty poset connected. Let~$\mcO \in \J(P)$ and let~$p\in P$ be any element. We say~$p$ can be {\em toggled in} to~$\mcO$ if $p$ is a minimal element not in $\mcO$, and that $p$ can be {\em toggled out} of $\mcO$ if $p$ is a maximal element in $\mcO$. Equivalently, $p$ can be toggled in to $\mcO$ if~$p \notin \mcO$ and~$\mcO\cup\{p\}\in\J(P)$, and~$p$ can be toggled out of $\mcO$ if~$p \in \mcO$ and~$\mcO\setminus\{p\} \in \J(P)$.

\begin{definition}
The {\em jaggedness} of an order ideal $\mcO \in \J(P)$, denoted $\jag(\mcO)$, is the total number of elements $p\in P$ which can be toggled into $\mcO$ or out of $\mcO$.  
\end{definition}

In this paper we will consider $\J(P)$ as a discrete probability space and so will refer to functions on $\J(P)$ as random variables. Define, for each $p \in P$, two indicator random variables $\T^+_p,\T^-_p\col \J(P) \ra \RR$ that record whether $p$ is toggleable-in (respectively toggleable-out) of an order ideal.  Explicitly, for $\mcO\in\J(P)$, we define
\begin{align*}
\T^+_p(\mcO) &\coleq \begin{cases}
1 & \text{ if $p$ can be toggled in to $\mcO$}, \\
0 & \text{ otherwise}
\end{cases}, \\
\T^-_p(\mcO) &\coleq \begin{cases}
1 & \text{ if $p$ can be toggled out of $\mcO$}, \\
0 & \text{ otherwise}.
\end{cases}
\end{align*}
These random variables are highly related to Striker's toggleability~\cite[Definition 6.1]{striker2015toggle}. Indeed, her toggleability statistic $\T_p$ simply decomposes as $\T_p = \T_p^+ - \T_p^-$. Note furthermore that $\jag = \sum_{p \in P} (\T_p^+ + \T_p^-)$. In this paper, we will show how certain conditions on $\T_p$ imply conditions on $\jag$, as in the following main definition of the section.

\begin{definition}
Let $\mu$ be a probability distribution on $\J(P)$.  Given an element $p\in P$, we say that $\mu$ is {\em toggle-symmetric at $p$} if
\[\PP_\mu(\,p \text{ can be toggled in to }\mcO) = \PP_\mu(\,p \text{ can be toggled out of }\mcO).\]
Equivalently, $\mu$ is toggle-symmetric at $p$ if $$\EE_{\mu}(\T_p) = \EE_{\mu}(T^{+}_p) - \EE_\mu(T^{-}_p) =  0.$$ We say that~$\mu$ is {\em toggle-symmetric} if it is toggle-symmetric at every $p \in P$.  
\end{definition}

We would like to introduce toggle-symmetric probability distributions as an interesting class of distributions on order ideals of posets.  We now give plenty of good examples of toggle-symmetric distributions.  Throughout, we fix a poset $P$ with $n \coleq \#P$.  

\subsection{Toggle-symmetric distributions arising from \texorpdfstring{$P$}{P}-partitions} \label{sec:pparts}

In this subsection we define several families of toggle-symmetric distributions that arise from $P$-partitions and related objects like linear extensions. For background on $P$-partitions see~\cite[\S3.15]{stanley2012ec1} or the recent historical survey~\cite{gessel2015historical}.

\begin{definition}
A \emph{linear extension} of $P$ is a bijection $\ell\col P \to \{1,2,\ldots,n\}$ such that~$p \le q$ for $p,q \in P$ implies $\ell(p) \le \ell(q)$. The {\em linear distribution} $\mu_\lin$ on $\J(P)$ is defined as follows: for $\mcO \in \J(P)$ we define $\mu_{\lin}(\mcO)$ to be the probability that, choosing a linear extension $\ell$ of $P$ and $k \in \{0,1,2,\ldots,n\}$ uniformly at random, the order ideal~$\ell^{-1}(\{1,\ldots,k\})$ is equal to~$\mcO$.
\end{definition}

\begin{definition}
A \emph{weak reverse $P$-partition of height $m$} is a map $\ell\col P \to \{0,1,\ldots,m\}$ such that~$p \le q$ for $p,q \in P$ implies $\ell(p) \le \ell(q)$. Fix $m\ge 1.$  The {\em weak  distribution}~$\mu_{m,\le}$ on $\J(P)$ is defined as follows: for $\mcO \in \J(P)$ we define $\mu_{m,\le}(\mcO)$ to be the probability that, choosing a weak reverse $P$-partition $\ell$ of height $m$ and $k \in \{1,2,\ldots,m\}$ uniformly at random, the order ideal~$\ell^{-1}(\{0,\ldots,k-1\})$ is equal to~$\mcO$.
\end{definition}

\begin{remark} \label{rem:uniform} 
There is a bijection between $\J(P)$ and the set of weak reverse $P$-partitions of height $1$ given by sending an order ideal $\mcO$ to  $\Ind -\Ind_{\mcO}$ where $\Ind_X$ is the indicator function of a subset $X \subseteq P$. Thus $\mu_{1,\le}$ is simply the {\em uniform} distribution on~$\J(P)$, which we will denote $\mu_\unif$.
\end{remark}

\begin{definition}
A \emph{strict reverse $P$-partition of height $m$} is a map $\ell\col P \to \{0,1,\ldots,m\}$ such that~$p < q$ for $p \neq q \in P$ implies $\ell(p) < \ell(q)$. The \emph{rank} of $P$, denoted~$\rk(P)$, is the maximum length of a chain of~$P$. Given~$m\ge \rk(P)$, the {\em strict distribution} $\mu_{m,<}$ on~$\J(P)$ is defined as follows: for~$\mcO \in \J(P)$ we define $\mu_{m,<}(\mcO)$ to be the probability that, choosing a strict reverse $P$-partition $\ell$ of height $m$ and~$k \in \{0,1,\ldots,m+1\}$ uniformly at random, the order ideal~$\ell^{-1}(\{0,\ldots,k-1\})$ is equal to~$\mcO$.
\end{definition}

\begin{remark} \label{rem:ranked}
We say that $P$ is \emph{ranked} if there exists a \emph{rank function} $\rk\col P \to \ZZ_{\ge 0}$ such that $\rk(q) = \rk(p) + 1$ whenever $q$ covers $p$ (denoted $p \lessdot q$) in $P$. We will always assume that~$0$ is in the image of the rank function in which case~$\rk$ is uniquely determined if it exists. We say that~$P$ is \emph{graded} if it is ranked and moreover~$\rk(p) = 0$ for all minimal elements~$p$ of~$P$ and~$\rk(q) = \rk(P)$ for all maximal elements $q$ of $P$. Equivalently,~$P$ is graded if all maximal chains of $P$ have the same length. When $P$ is graded, the distribution $\mu_{\rk(P),<}$ is easy to describe: it is uniform on the set of all order ideals~$\rk^{-1}(\{0,\ldots,k-1\})$ for~$k \in \{0,\ldots,\rk(P)+1\}$. In this case we call $\mu_{\rk(P),<}$ the \emph{rank distribution} and denote it by~$\mu_{\rk}$.
\end{remark}

Now we will show that all of the above distributions $\mu_\lin, \mu_{m,\le},$ and $\mu_{m,<}$ are toggle-symmetric.  We start by proving toggle-symmetry for $\mu_{m,\le}$ and $\mu_{m,<}$.

\begin{lemma} \label{lem:ppartsym}
For any poset $P$ and any $m\ge 1$, the distribution $\mu_{m,\le}$ on~$\J(P)$ is toggle-symmetric. Similarly, for any $m\ge \rk(P)$, the distribution $\mu_{m,<}$ on~$\J(P)$ is toggle-symmetric. In particular, the uniform distribution $\mu_{\unif}$ is toggle-symmetric and, if~$P$ is graded, the rank distribution $\mu_{\rk}$ is toggle-symmetric.
\end{lemma}

\begin{proof}
Let us start by proving the lemma with the weak distribution $\mu_{m,\le}$.  For a given~$p\in P$, we will define an involution~$\tau_p$ on the set of weak reverse $P$-partitions of height $m$, and this involution will verify that $\mu$ is toggle-symmetric at $p$. Let $\widehat{P}$ denote the poset obtained from $P$ by adjoining a minimal element $\widehat{0}$ and a maximal element $\widehat{1}$. Let $\ell$ be a weak reverse $P$-partition of height $m$; we extend $\ell$ to $\widehat{P}$ by setting $\ell(\widehat{0}) \coleq 0$ and~$\ell(\widehat{1}) \coleq m$. Then for $p,q \in P$ we define
\begin{equation} \label{eq:ppartiontoggle}
\tau_p(\ell) (q) \coleq \begin{cases} \ell(q) & \textrm{$q \neq p$}, \\ \mathrm{max}\{\ell(r)\col r \lessdot p,r \in \widehat{P}\} + \mathrm{min}\{\ell(r)\col p \lessdot r, r \in \widehat{P}\} - \ell(p) & \textrm{$q=p$}. \end{cases}
\end{equation}
Evidently $\tau_p$ is an involution and preserves the relevant weak inequalities. For $\ell$ a weak reverse $P$-partition of height $m$ and $k \in \{1,2,\ldots,m \}$ we have
\[ \T^{+}_p (\ell^{-1}(\{0,\ldots,k-1\})) = \begin{cases} 1 &\textrm{if $\mathrm{max}\{\ell(r)\col r \lessdot p,r \in \widehat{P}\} < k \leq \ell(p)$},\\ 0 &\textrm{otherwise}. \end{cases}\]
Thus
\[\EE_{\mu_{m,\leq}}(\T^{+}_p) = \EE(\ell(p) -\mathrm{max}\{\ell(r)\col r \lessdot p,r \in \widehat{P}\})\] 
for $\ell$ a uniformly random weak reverse $P$-partition of height $m$. Similarly, 
\[\EE_{\mu_{m,\leq}}(\T^{-}_p) = \EE(\mathrm{min}\{\ell(r)\col p \lessdot r, r \in \widehat{P}\} - \ell(p)).\] 
But then observe that
\begin{align*}
\EE_{\mu_{m,\leq}}(\T^{+}_p) &= \EE(\ell(p) -\mathrm{max}\{\ell(r)\col r \lessdot p,r \in \widehat{P}\}) \\
&= \EE(\mathrm{min}\{\ell(r)\col p \lessdot r, r \in \widehat{P}\} - \tau_p(\ell)(p)) \\
&= \EE_{\mu_{m,\leq}}(\T^{-}_p)
\end{align*}
and thus indeed $\mu_{m,\leq}$ is toggle symmetric.

The proof of the lemma for the strong distribution is exactly analogous to the weak distribution.  Let~$\ell$ be a strict reverse $P$-partitions of height $m$; we extend $\ell$ to $\widehat{P}$ by setting $\ell(\widehat{0}) \coleq -1$ and~$\ell(\widehat{1}) \coleq m+1$. For $p \in P$ we define an involution $\tau_p$ on the set of strict reverse~$P$-partitions by the exact same formula~(\ref{eq:ppartiontoggle}) as above. This involution again establishes that $\EE_{\mu_{m,<}}(\T_p) = 0$. The last sentence follows from Remarks~\ref{rem:uniform} and~\ref{rem:ranked}.
\end{proof}

\begin{prop} \label{prop:limit} We have
\[\lim_{m\to \infty} \mu_{m,\le} = \lim_{m\to \infty} \mu_{m,<} = \mu_\lin.\]
\end{prop}

\begin{proof}

To prove this proposition we will define an intermediary distribution $\mu_{m,\hookrightarrow}$ on~$\J(P)$ based on injective order-preserving maps. It will turn out that~$\mu_{m,\hookrightarrow} = \mu_{\lin}$. For~$m \geq n-1$ and~$I \in \J(P)$ we define $\mu_{m,\hookrightarrow}(I)$ to be the probability that, choosing an \emph{injective} order-preserving map~$\ell\colon P \to \{0,1,\ldots,m\}$ and $k \in \{0,1,\ldots,m+1\}$ uniformly at random, the order ideal $\ell^{-1}(\{0,\ldots,k-1\})$ is equal to $I$. First we claim
\[\lim_{m\to \infty} \mu_{m,\le} = \lim_{m\to \infty} \mu_{m,<} = \lim_{m\to \infty} \mu_{m,\hookrightarrow}.\]
This is clear because as $m \to \infty$ the fraction of (weak or strict) order-preserving maps~$P \to \{0,1,\ldots,m\}$ that are injective approaches $1$.  Furthermore, in the case of~$\mu_{m,\le},$ as $m \to \infty$ the probability that a uniformly chosen~$k \in \{0,\ldots,m+1\}$ actually lands in $\{1,\ldots,m\}$ approaches~$1$.

Next we claim that $\mu_{m,\hookrightarrow} = \mu_\lin$ for all $m \geq n-1$. 
Given an order ideal $I$, let~$L(I)$ denote the set of linear extensions $\ell\colon P\stackrel{\cong}{\longrightarrow} \{1,\ldots,n\}$ that are compatible with~$I$; that is, if $p\in I$ and $q\not\in I$ then $\ell(p) < \ell(q)$.  Let $\Phi_{n,m}$ denote the set of order-preserving injective maps
$\phi\colon \{0,\ldots,n\!+\!1\} \rightarrow \{-1,\ldots,m\!+\!1\}$
such that $\phi(0) = -1$ and~$\phi(n+1) = m+1.$  Then

$$\PP_{\mu_{m,\hookrightarrow}}(I) \,\, {\boldsymbol \sim} \, \sum_{\ell\in L(I)} \sum_{\phi \in \Phi_{n,m}}\! \left(\phi(\#I\!+\!1)-\phi(\#I)\right),$$
where the sign $\sim$ denotes proportionality up to a constant. The reason is that 
the order-preserving injective map $\phi|_{\{1,\ldots,n\}}\circ \ell$ gives rise to $I$ if and only if $I$ is compatible with $\ell$, and furthermore, if $I$ is compatible with the linear extension $\ell$, then $I$ arises from the map $\phi|_{\{1,\ldots,n\}}\circ \ell$ with probability proportional to $\phi(\#I\!+\!1)-\phi(\#I)$.

But now we claim the inner sum $\sum_{\phi \in \Phi_{n,m}}\! \left(\phi(\#I\!+\!1)-\phi(\#I)\right)$ is a constant, not depending on $\#I$.  Indeed, given $\phi\in\Phi_{n,m}$, call the {\em signature} of $\phi$ the multiset of consecutive differences
$$\{\phi(1)\!-\!\phi(0), \phi(2)\!-\!\phi(1),\ldots,\phi(n\!+\!1)\!-\!\phi(n)\}.$$
Then the sum $\sum \left(\phi(\#I\!+\!1)-\phi(\#I)\right)$ restricted to any given signature-equivalence class in $\Phi_{n,m}$ is a constant, not depending on $\#I$.  Therefore, the same is true for the sum over all $\phi\in\Phi_{n,m}$.  This shows that $\mu_{m,\hookrightarrow}(I)$ is proportional to $\#L(I)$ and so~$\mu_{m,\hookrightarrow} = \mu_{\lin}$.
\end{proof}

From Lemma~\ref{lem:ppartsym} and Proposition~\ref{prop:limit} we conclude:
\begin{cor}\label{cor:linsym}
The  linear distribution $\mu_\lin$ on $\J(P)$ is toggle-symmetric.
\end{cor}
\noindent So for any poset $P$ we have the following ``spectrum'' of toggle-symmetric distributions:
\[\xymatrixcolsep{5pc}\xymatrix{
\mu_{\unif} \ar[r]^{\mu_{m,\leq}} & \mu_{\lin} & \ar[l]_{\mu_{m,<}} \mu_{\rk}
}\]
where the rightmost distribution $\mu_{\rk}$ applies only to graded $P$.

\begin{remark}\label{r:shuffles}
Actually, we can give a much more direct and satisfying proof of Corollary~\ref{cor:linsym} which goes by way of defining some interesting involutions $\sigma_p$ for $p\in P$ on the set of linear extensions of $P$.  Note that Corollary~\ref{cor:linsym} was proved for skew shapes in~\cite[Lemma 2.9]{chan2015genera}, using involutions on pairs (linear extension, order ideal).  Instead, the involutions $\sigma_p$ that we use here can be regarded as ``shuffle" operations on the set of linear extensions only, in the spirit of \cite{ayyer2014spectral}.  
They are defined as follows. Let~$\ell$ be a linear extension of~$P$. Extend $\ell$ to $\widehat{P}$ by setting $\ell(\widehat{0}) \coleq 0$ and $\ell(\widehat{1}) \coleq n+1$. Let~$p \in P$ and~$x \coleq \ell(p)$ and~$x' \coleq \mathrm{max}\{\ell(r)\col r \lessdot p,r \in \widehat{P}\} + \mathrm{min}\{\ell(r)\col p \lessdot r, r \in \widehat{P}\} - x$. Then for $q \in P$, define
\[ \sigma_p(\ell)(q) \coleq \begin{cases} x' &\textrm{if $\ell(q) = x$}, \\ \ell(q) - 1 &\textrm{if $x < \ell(q)\leq x'$}, \\ \ell(q) + 1 &\textrm{if $x'\leq \ell(q)<x$}, \\ \ell(q) &\textrm{otherwise}. \end{cases}\]
It is straightforward to verify that $\sigma_p$ is indeed an involution on the set of linear extensions of $P$, and that, exactly analogously to the proof of Lemma~\ref{lem:ppartsym}, this involution verifies that $\EE_{\mu_{\lin}}(\T_p) = 0$.  
\end{remark}

\subsection{Toggle-symmetric distributions arising from the toggle group} \label{sec:togglegpsym}

The toggle group was introduced by Cameron and Fon-der-Flaass~\cite{cameron1995orbits} in order to study a certain combinatorial map on order ideals that is now called \emph{rowmotion}. For background on the toggle group, rowmotion, and gyration, see~\cite{striker2012promotion}. For $p \in P$ we define the \emph{toggle at~$p$}, denoted~$\tau_p\colon \J(P) \to \J(P)$, by
\[\tau_p(\mcO) \coleq \begin{cases} \mcO \Delta \{p\}&\textrm{if $\mcO \Delta \{p\} \in \J(P)$}, \\
\mcO &\textrm{otherwise}\end{cases}\]
where $\Delta$ denotes the symmetric difference. These are the same $\tau_p$ as defined in the proof of Lemma~\ref{lem:ppartsym} for the weak distribution when $m=1$ via the bijection mentioned in Remark~\ref{rem:uniform}. The \emph{toggle group} is the subgroup of the permutation group $\mathfrak{S}_{\J(P)}$ generated by all toggles $\tau_p$ for $p \in P$. Recently Striker~\cite{striker2015toggle} proved that certain distributions on $\J(P)$ arising from toggle group elements are toggle-symmetric.

\begin{definition}
\emph{Rowmotion} is the element of the toggle group 
\[\tau_{\ell^{-1}(1)} \circ \tau_{\ell^{-1}(2)} \cdots \circ \tau_{\ell^{-1}(n)} \]
where $\ell$ is any linear extension of $P$. Note that because $\tau_p$ and $\tau_q$ commute unless $p \lessdot q$ or $q \lessdot p$ this composition indeed gives a well-defined map.
\end{definition}

\begin{definition}
Assume $P$ is ranked. \emph{Gyration} is the element of the toggle group
\[\tau_{o_1} \circ \tau_{o_2} \circ \cdots \circ \tau_{o_{n_1}} \circ \tau_{e_1} \circ \tau_{e_2} \circ \cdots \circ \tau_{e_{n_0}} \]
with $\{e_1,\ldots,e_{n_0}\} = \{p \in P\colon \rk(p)\textrm{ is even}\}$ and $\{o_1,\ldots,o_{n_1}\} = \{p \in P\colon \rk(p)\textrm{ is odd}\}$. Again, because most toggles commute, this composition gives a well-defined map.
\end{definition}

\begin{thm}[Striker~\cite{striker2015toggle}] \label{thm:striker}
Let $P$ be a poset and let $\varphi \colon \J(P)\to\J(P)$ be rowmotion or, in the case where $P$ is ranked, gyration. Then the distribution $\mu$ that is supported uniformly on a fixed $\varphi$-orbit $\Orb$ is toggle-symmetric.
\end{thm}

Actually, Striker phrased her result in the language of homomesy. Homomesy is a certain phenomenon in dynamical algebraic combinatorics, recently introduced by Propp and Roby~\cite{propp2013homomesy}, concerning statistical averages along orbits of combinatorial maps.

\begin{definition}
Let $\mathcal{S}$ be a set of combinatorial objects and $\varphi\colon \mathcal{S} \to \mathcal{S}$ an invertible map. We say that the statistic $f:\mathcal{S} \to \RR$ is~$c$-\emph{mesic} with respect to the action of $\varphi$ on~$\mathcal{S}$ if there is $c \in \RR$ such that~$\frac{1}{\#\Orb} \sum_{s \in \Orb} f(s) = c$ for each $\varphi$-orbit $\Orb$. In other words, we say $f$ is \emph{homomesic} with respect to $\varphi$ if the average of $f$ is the same for each $\varphi$-orbit.
\end{definition}

What Striker proved was that, for any poset $P$ and any $p \in P$, the signed toggleability statistic $\T_p$ is $0$-mesic with respect to rowmotion (\cite[Lemma 6.2]{striker2015toggle}) and is also $0$-mesic with respect to gyration when $P$ is ranked (\cite[Theorem 6.7]{striker2015toggle}). Clearly these results are equivalent to Theorem~\ref{thm:striker} as stated above.

\section{The expected jaggedness in skew shapes}\label{sec:exp}
In this section, we prove a general result giving a formula for the expected jaggedness of an order ideal in a poset $P$ for any toggle-symmetric distribution whenever~$P$ is the poset corresponding to a skew Young diagram. 

A partition $\lambda = (\lambda_1\ge\ldots\ge\lambda_k)$ is a sequence of weakly decreasing positive integers.  Recall that associated to~$\lambda$ is a {\em Young diagram} consisting of $\lambda_i$ boxes in the~$i^{th}$ row, left-justified.  Given two partitions $\lambda$ and $\nu$ of the numbers $\ell$ and $n$ respectively, we say that $\nu \subseteq \lambda$ if $\nu_i \le \lambda_i$ for all $i$.  We use the usual convention that $\lambda_i = 0$ if $i$ is greater than the number of parts 
of $\lambda$. We use English notation when drawing partitions, so for instance the Young diagram corresponding to the partition $\lambda = (4,3)$ is
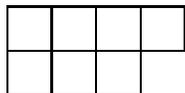
\[\ydiagram{4,3}\]

\begin{definition} Let $\nu \subseteq \lambda$ be two partitions.  The diagram obtained by subtracting the Young diagram of $\nu$ from the Young diagram of $\lambda$ is called a {\em skew Young diagram} or {\em skew shape}.  We will write $\sigma = \lambda/\nu$ for this shape.
\end{definition}

Let $\sigma = \lambda/\nu$ be a skew shape.   Throughout, we let $a$ denote the {\em height} of $\sigma$, i.e.,~the number of rows in $\sigma$, and let $b$ denote the {\em width} of $\sigma$, i.e.,~the number of columns.  In order to refer to the boxes of $\sigma$ and their corners, we will fix coordinates as follows.  Place $\sigma$ in an $a\times b$ rectangle.  Our convention will be that the northwest corner of the rectangle is $(0,0)$ and the southeast corner is~$(a,b)$.  The corners of the boxes of~$\sigma$ are then various lattice points in this rectangle. Furthermore, we will extend this coordinate system to the boxes of $\sigma$ by writing $[i,j]$ for the box whose {\em southeast} corner is $(i,j)$.  For example, the upper-leftmost box of a Young diagram is the box $[1,1]$.

Associated to any skew shape $\sigma$ is a poset $P_{\sigma}$ whose elements are the boxes of $\sigma$ and with $[i,j] \leq [k,l]$ if and only if $i \leq k$ and $j \leq l$. Note that $P_{\sigma}$ is always ranked (where the rank function $\rk([i,j]) = i + j - \kappa$, for an appropriate constant $\kappa$, records the diagonal) but is not always graded. All of the general poset-theoretic constructions from Section~\ref{sec:pparts} have more common names when specialized to skew shapes, which we record in the following dictionary:
\begin{center}
\begin{tabular}{c | c}
Poset $P_{\sigma}$ & Skew shape $\sigma$ \\
\hline
Order ideals & Subshapes \\
\hline
Linear extensions & Standard Young tableaux \\
\hline
Weak reverse $P_{\sigma}$-partitions & Reverse plane partitions \\
\hline
Strict reverse $P_{\sigma}$-partitions & Increasing tableaux
\end{tabular}
\end{center}
We will not go through all of these terms in detail, but let us comment for a moment on subshapes. If $\sigma = \lambda/\nu$ is a skew shape, then a \emph{subshape of $\sigma$} is a skew shape $\rho / \nu$ where~$\rho$ is a partition satisfying~$\nu \subseteq \rho \subseteq \lambda$. These are clearly the same as order ideals of $P_{\sigma}$ and so we use the notation $\J(\sigma)$ for the set of subshapes of $\sigma$. We also often identify a subshape $\rho / \nu \in \J(\sigma)$ with its \emph{lattice path}, which is the sequence of steps of the form $(-1,0)$ and $(0,1)$ connecting the point $(a,0)$ to $(0,b)$ (in the coordinate system defined above) given by the southeast border of $\rho$. In this way $\J(\sigma)$ is in bijection with the set of lattice paths connecting $(a,0)$ to $(0,b)$ that stay within the diagram of~$\sigma$. For an example of this bijection see Figure~\ref{fig:latticepath}.

\begin{figure}
\begin{center}
\begin{tikzpicture}
\node at (0,0) {\begin{ytableau} \none & *(yellow) &   \\ *(yellow) &  &  \\ *(yellow) \end{ytableau}};
\def\x{0.6}
\draw[red,ultra thick] (-1.5*\x,-1.5*\x) -- (-0.5*\x,-1.5*\x) -- (-0.5*\x,0.5*\x) -- (0.5*\x,0.5*\x) -- (0.5*\x,1.5*\x) -- (1.5*\x,1.5*\x);
\end{tikzpicture}
\end{center}
\caption{With $\sigma = (3,3,1) / (1)$, we depict $(2,1,1) / (1) \in \J(\sigma)$ shaded in yellow and its associated lattice path in red.} \label{fig:latticepath}
\end{figure}

\begin{definition}\label{def:outward}
Let $\sigma$ be a skew shape. We say $\sigma$ is \emph{connected} if the poset $P_{\sigma}$ is connected. Suppose $\sigma$ is connected. Then an {\em outward corner} of $\sigma$ is two consecutive steps along the boundary of $\sigma$ that do not belong to the same line and do not border a common box of $\sigma$.  We say that a corner \emph{occurs} at the lattice point $(i,j)$ where its two steps meet. We write~$C(\sigma)$ for the set of outward corners of $\sigma$.
\end{definition}

Note that because $\sigma$ is a skew shape, the outward corners of $\sigma$ are either {\em northwest corners} or {\em southeast corners}, i.e.,~they comprise part of the northwest border of $\sigma$ or the southeast border, respectively.

The following notation will be convenient for us: given a box $[i,j]\in\sigma$, we define
\[C_{ij}(\sigma) = \{\text{corners $c\in C(\sigma)$ occurring strictly northwest or strictly southeast of $[i,j]$}\}.\]
When we say that a corner $c$ occurs ``strictly northwest'' or ``strictly southeast'' of a box~$[i,j]$, we mean that it occurs strictly northwest (respectively strictly southeast) of the {\em center} of that box.  For example, a corner at the point $(i,j)$ occurs strictly southeast of the box~$[i,j]$. Figure~\ref{fig:corners} illustrates our notation for corners. 

For $c \in C(\sigma)$ and $\mu$ a probability distribution on $\J(\sigma)$ we use the notation $\PP_{\mu}(c)$ to mean the probability with respect to $\mu$ that a subshape of~$\sigma$, thought of as a lattice path, includes the two steps of the corner $c$. It is important to note that if the corner~$c \in C(\sigma)$ occurs at~$(i,j)$, then saying that the lattice path $\rho \in \J(c)$ includes~$c$ is a stronger statement than merely saying that $\rho$ passes through $(i,j)$.

\begin{figure}
\begin{tikzpicture}
\node at (0,0) {\begin{ytableau} \none & \none & \none &  &  &    &  \\ \none & \none & \none &  & && \\ \none & \none &  & *(yellow)   & \\ &  &  &\\ \end{ytableau}};
\def\x{0.6}
\fill (-0.5*\x,0*\x) circle (0.1cm);
\fill[white] (-0.5*\x,0*\x) circle (0.05cm);
\fill (-1.5*\x,-1*\x) circle (0.1cm);
\fill (1.5*\x,0*\x) circle (0.1cm);
\fill (0.5*\x,-1*\x) circle (0.1cm);
\fill[white] (0.5*\x,-1*\x) circle (0.05cm);
\end{tikzpicture}
\caption{A diagram explaining our notation for corners. Box~$[3,4]$ of $\sigma$ is shaded yellow and the points where corners~$c \in C(\sigma)$ occur are marked with a circle; the points where~$c \in C_{34}(\sigma)$ occur are open circles.} \label{fig:corners}
\end{figure}

\begin{definition}\label{def:disp}
Let $\sigma$ be a connected skew shape with height $a$ and width $b$. The \emph{main anti-diagonal of $\sigma$} is the line joining $(a,0)$ to $(0,b)$. For $(i,j) \in \RR^2$ let $\vec{d}(i,j)$ denote the vector from $(i,j)$ to the main anti-diagonal of $\sigma$ (and orthogonal to it). For an outward corner~$c \in C(\sigma)$ that occurs at $(i,j)$ we define the \emph{displacement} of $c$ to be
\[\delta(c) \coleq \begin{cases} \textrm{the unique $x\in \RR$ with $\vec{d}(0,0) = x \cdot \vec{d}(i,j)$} &\textrm{if $c$ is a northwest corner}, \\
\textrm{the unique $x\in \RR$ with $\vec{d}(a,b) = x \cdot \vec{d}(i,j)$} &\textrm{if $c$ is a southeast corner}. \end{cases}\]
Note that $\delta(c)$ is a signed quantity. Explicitly,
\[\delta(c) = \begin{cases} 1 - \frac{i}{a} - \frac{j}{b} &\textrm{if $c$ is a northwest corner}, \\
-1 + \frac{i}{a} + \frac{j}{b} &\textrm{if $c$ is a southeast corner}. \end{cases}\]
\end{definition}

Now we can state main theorem of our paper, which computes the expected jaggedness of a subshape of $\sigma$ for any toggle-symmetric distribution as the harmonic mean of its height and width, up to a sum of signed correction terms.

\begin{thm}\label{thm:main}
Let $\sigma$ be a connected skew shape with height $a$ and width $b$.  Let $\mu$ be any toggle-symmetric probability distribution on $\J(\sigma)$.  Then the expected jaggedness of a subshape of $\sigma$ with respect to the distribution $\mu$ is
\begin{equation}\label{e:main}
\EE_\mu(\jag) = 
\frac{2ab}{a+b}\left( 1 + \sum_{c\in C(\sigma)} \!\delta(c)\,\PP_\mu(c)\right).
\end{equation}
\end{thm}

In the rest of this section we will prove Theorem \ref{thm:main}. In order to that, we define a set of random variables $R_{ij}$ that we refer to as \emph{rooks}. The proof of the main theorem involves strategically placing rooks on our skew shape $\sigma$. 

Let $[i,j]$ be a box in~$\sigma$.  We write $\T_{ij}^+$ and $\T_{ij}^-$ for the toggle-indicator random variables $T_{[i,j]}^{+}$ and $T_{[i,j]}^{-}$ on $\J(P_{\sigma}) = \J(\sigma)$ defined in Section~\ref{sec:togglesym}.  We define the rook random variable~$R_{ij}\colon \J(\sigma) \to \RR$ as follows:
\begin{equation}\label{eq:Rij}
R_{ij} \coleq \sum_{\substack{i'\le i,\,j'\le j \\ [i',j']\in\sigma }}\!\T_{i'j'}^+ 
+ \sum_{\substack{i'\ge i,\,j'\ge j \\ [i',j']\in\sigma }}\!\T_{i'j'}^- 
- \sum_{\substack{i'< i,\,j'< j \\ [i',j']\in\sigma }}\!\T_{i'j'}^- 
- \sum_{\substack{i'> i,\,j'> j \\ [i',j']\in\sigma }}\!\T_{i'j'}^+.
\end{equation}
The equation defining $R_{ij}$ is complicated and it is best understood by a drawing as in Figure~\ref{fig:rook}.  In this figure, we record the coefficients of the terms $T_{i'j'}^+$ and $T_{i'j'}^-$ in $R_{ij}$ in the northwest and southeast corners, respectively, of the box~$[i',j']$. The reason we call~$R_{ij}$ a rook is explained by the next lemma, which says that for a toggle-symmetric distribution $\mu$ only the toggleability statistics corresponding to boxes in the same row or column as $[i,j]$ contribute to the expectation $\EE_{\mu}(R_{ij})$.

\begin{figure}
{\ytableausetup{boxsize=2em}\begin{ytableau}
\scalebox{1.0}[1.1]{$^{1}\,\,_{\text{-}1}$} & \scalebox{1.0}[1.1]{$^{1}\,\,\,\,\,\,$} & & & & & & \\
\scalebox{1.0}[1.1]{$^{1}\,\,_{\text{-}1}$} & \scalebox{1.0}[1.1]{$^{1}\,\,\,\,\,\,$} & & & & & \\
\scalebox{1.0}[1.1]{$^{1}\,\,\,\,\,\,$} & \scalebox{1.0}[1.1]{$^{1}\,\,\,_{1}$} & \scalebox{1.0}[1.1]{$\,\,\,\,\,\,_{1}$} & \scalebox{1.0}[1.1]{$\,\,\,\,_{1}$} & \scalebox{1.0}[1.1]{$\,\,\,\,_{1}$} & \scalebox{1.0}[1.1]{$\,\,\,\,_{1}$}\\
 & \scalebox{1.0}[1.1]{$\,\,\,\,\,\,_{1}$} & \scalebox{1.0}[1.1]{$^{\text{-}1}\,\,_{1}$} & \scalebox{1.0}[1.1]{$^{\text{-}1}\,\,_{1}$} & \scalebox{1.0}[1.1]{$^{\text{-}1}\,\,_{1}$} & \scalebox{1.0}[1.1]{$^{\text{-}1}\,\,_{1}$} \\
  & \scalebox{1.0}[1.1]{$\,\,\,\,\,\,_{1}$} & \scalebox{1.0}[1.1]{$^{\text{-}1}\,\,_{1}$} & \scalebox{1.0}[1.1]{$^{\text{-}1}\,\,_{1}$} & \scalebox{1.0}[1.1]{$^{\text{-}1}\,\,_{1}$} & \scalebox{1.0}[1.1]{$^{\text{-}1}\,\,_{1}$} 
\end{ytableau}\ytableausetup{boxsize=1.5em} }
\caption{An example of a ``rook'' at the box $[3,2]$.  }\label{fig:rook}
\end{figure}

\begin{lemma} \label{lem:rook} Let $\sigma$ be a skew shape and $\mu$ a toggle-symmetric probability distribution on $\J(\sigma)$. Then for any $[i,j] \in \sigma$ we have 
\[\EE_{\mu}(R_{ij}) = \sum_{[i',j] \in \sigma} \EE_{\mu}(\T^{+}_{i',j}) + \sum_{[i,j'] \in \sigma} \EE_{\mu}(\T^{+}_{i,j'}).\] 
\end{lemma}

\begin{proof}
 Expanding formula~(\ref{eq:Rij}),
 \[\begin{array}{r l }
R_{ij}=&  \sum_{\substack{i'< i,\,j'< j \\ [i',j']\in\sigma }}\!\T_{i'j'}^+  - 
 \sum_{\substack{i'< i,\,j'< j \\ [i',j']\in\sigma }}\!\T_{i'j'}^- +
 \sum_{\substack{i'> i,\,j'> j \\ [i',j']\in\sigma }}\!\T_{i'j'}^+ -
  \sum_{\substack{i'> i,\,j'> j \\ [i',j']\in\sigma }}\!\T_{i'j'}^- \\[.5cm]
  & + \sum_{[i,j'] \in \sigma}\T^{+}_{i',j} + \sum_{[i,j'] \in \sigma}\T^{+}_{i,j'}.\\
\end{array}\]
Since $\mu$ is a toggle symmetric distribution by linearity of expectation we get
\[\EE_{\mu}\left(\sum_{\substack{i'< i,\,j'< j \\ [i',j']\in\sigma }}\T_{i'j'}^+  - 
 \sum_{\substack{i'< i,\,j'< j \\ [i',j']\in\sigma }}\T_{i'j'}^-\right)=0; \qquad 
 \EE_{\mu}\left(\sum_{\substack{i'> i,\,j'> j \\ [i',j']\in\sigma }}\T_{i'j'}^+  - 
 \sum_{\substack{i'> i,\,j'> j \\ [i',j']\in\sigma }}\T_{i'j'}^-\right)=0.\] 
 Hence the claimed expression for $\EE_{\mu}(R_{ij})$ indeed holds.
\end{proof}

\begin{figure}
\begin{tikzpicture}
\node at (-0.07,0) {\ytableausetup{boxsize=2.5em} \begin{ytableau}
\none & \scalebox{1.1}[1.2]{$^{1}\,\,\,\,\,\,$} & & & & \\
\scalebox{1.1}[1.2]{$^{1}\,\,_{\text{-}1}$} & \scalebox{1.1}[1.2]{$^{1}\,\,\,\,\,\,$} & &  \\
\scalebox{1.1}[1.2]{$^{1}\,\,\,\,\,\,$} & \scalebox{1.1}[1.2]{$^{1}\,\,\,_{1}$} & \scalebox{1.1}[1.2]{$\,\,\,\,\,\,_{1}$} \\
 & \scalebox{1.1}[1.2]{$\,\,\,\,\,\,_{1}$}
\end{ytableau} \ytableausetup{boxsize=1.5em}};
\def\x{0.98}
\draw[blue,ultra thick] (-3*\x,-2*\x) -- (-3*\x,0*\x) -- (-2*\x,0*\x) -- (-2*\x,2*\x) -- (3*\x,2*\x);
\draw[red,ultra thick] (-3*\x,-2*\x) -- (-1*\x,-2*\x) -- (-1*\x,-1*\x) -- (0*\x,-1*\x) -- (0*\x,1*\x) -- (3*\x,1*\x) -- (3*\x,2*\x);
\fill (-1*\x,-1*\x) circle (0.1cm);
\fill (-2*\x,1*\x) circle (0.1cm);
\end{tikzpicture}
\caption{This figure illustrates how each subshape may contribute to~$\EE(R_{ij})$. Here $[i,j] = [3,2]$ and the points where corners $c \in C_{ij}$ occur are marked with a circle. Two lattice paths~$\rho_1,\rho_2\in \J(\sigma)$ are drawn in blue and red; we can verify that $R_{ij}(\rho_{k}) = 1 + \#C_{ij}(\rho_{k})$ for $k = 1,2$.}\label{fig:toll}
\end{figure}
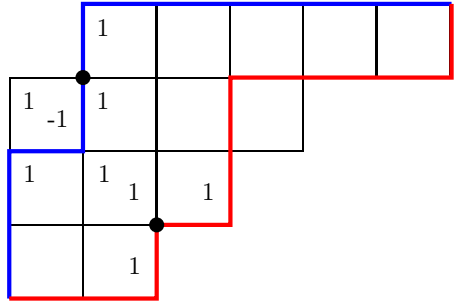

\begin{lemma} \label{lem:corners} Let $\sigma$ be a connected skew shape and $\mu$ a probability distribution on $\J(\sigma)$. Then for any $[i,j] \in \sigma$ we have 
\[\EE_{\mu}(R_{ij}) = 1 + \sum_{c \in C_{ij}(\sigma)}\PP_{\mu}(c).\]
\end{lemma}
\begin{proof}
Let $\rho \in \J(\sigma)$. Let $C_{ij}(\rho)$ be the set of all corners $c\in C_{ij}(\sigma)$ included in the lattice path $\rho$. We observe that $R_{ij}(\rho)=1+\#C_{ij}(\rho)$.  This observation is again best understood by a picture, as in Figures \ref{fig:rook} and \ref{fig:toll}.  
In Figure~\ref{fig:rook}, the set $C_{ij}(\sigma)$ is empty, and the claim that $R_{ij}(\rho) = 1$ for any lattice path $\rho$ drawn through the skew shape corresponds to the observation that the turns in $\rho$ always have total weight 1 (with the weights as drawn).  As usual, we identify lattice paths and subshapes.  

The more general formula $R_{ij}(\rho) = 1 + \#C_{ij}(\rho)$ then corresponds to the fact that any outward corner $c\in C_{ij}(\sigma)$ used by $\rho$ is no longer labeled $-1$, simply because there is no box at $c$ to be toggled in or toggled out.  This is illustrated in Figure~\ref{fig:toll}.

But
\[ \EE_{\mu}(\#C_{ij}(\rho)) = \sum_{c \in C_{ij}(\sigma)}\PP_{\mu}(c)\]
and hence the claimed expression for~$\EE_{\mu}(R_{ij})$ indeed holds.
\end{proof}

\begin{lemma}\label{lem:rookplacement}
For any connected skew shape $\sigma$ with height $a$ and width $b$ there exist integral coefficients~$r_{ij} \in \ZZ$ for $[i,j] \in \sigma$ such that 
\begin{itemize}
\item for all $1 \leq i \leq a$, $\sum_{[i,j'] \in \sigma} r_{i,j'} = b$;
\item for all $1 \leq j \leq b$, $\sum_{[i',j] \in \sigma} r_{i',j} = a$.
\end{itemize}
\end{lemma}

\begin{proof}
If we interpret the coefficient $r_{ij}$ as the number (possibly negative) of rooks placed at box $[i,j] \in \sigma$, the equalities say that each row should be attacked by a total of~$b$ rooks and each column by a total of~$a$ rooks. There are many possible such placements. Here is one. Let $\B\coleq\{[i,j]\in\sigma \col [i{+}1,j{+}1]\notin \sigma\}$ denote the set of boxes in the \emph{southeast border strip of $\sigma$}. We claim there is a unique choice of~$r_{ij}$ satisfying the desired equalities with $r_{ij} = 0$ if $[i,j] \notin \B$. Let $b_1,b_2,\ldots,b_m$ be the elements of~$\B$ in the unique order so that $b_1$ is southwesternmost, $b_m$ is northeasternmost, and $b_k$ is adjacent to $b_{k+1}$ for all $1 \leq k < m$. Then for each $1 \leq k \leq m$, exactly one of the following holds:
\begin{enumerate}[label=(\Roman*)]
\item $b_l$ is not in the same row as $b_k$ for all $l > k$;
\item $b_l$ is not in the same column as $b_k$ for all~$l > k$.
\end{enumerate}
Thus for $k=1,\ldots,m$ with $b_k = [i_k,j_k]$, we can choose the corresponding coefficients~$r_{i_k,j_k}$ in order: when we are in case~(I) we choose $r_{i_k,j_k}$ so that $\sum_{[i_k,j] \in \sigma} r_{i_k,j} = b$; when we are in case~(II) we choose $r_{i_k,j_k}$ so that $\sum_{[i,j_k] \in \sigma} r_{i,j_k} = a$. For each row or column, there is at least one $b_k$ in that row or column, so in the end all the equations will be satisfied. The result is an assignment of coefficients that looks like Figure~\ref{fig:rookplacement}.
\end{proof}

\begin{figure}
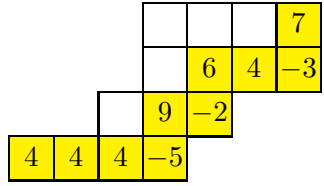

\begin{center} 
\begin{ytableau}
\none & \none & \none &  &  &    &*(yellow)7  \\
\none & \none & \none &  & *(yellow) 6&*(yellow)4 &*(yellow)-3 \\
\none & \none &  &  *(yellow)9 &*(yellow)-2  \\
*(yellow)4&*(yellow)4  & *(yellow)4 &*(yellow)-5 \\
\end{ytableau}
\end{center}
\caption{An example of a rook placement that satisfies Lemma~\ref{lem:rookplacement}. Here~$a=4$ and $b=7$; the southeast border strip is shaded in yellow.} \label{fig:rookplacement}
\end{figure}

\begin{proof}[Proof of Theorem~\ref{thm:main}]
Let $r_{ij}$ be the coefficients from Lemma~\ref{lem:rookplacement}. Note that the sum of all coefficients is $\sum_{[i,j]\in \sigma}r_{ij}=ab$. Also, for any $[i',j']\in \sigma$ the sum of coefficients in its row and its column is~$\sum_{\substack{[i,j]\in \sigma\\ i=i'}}r_{ij} +\sum_{\substack{[i,j]\in \sigma\\ j=j'}}r_{ij}=a+b$. Using Lemma~\ref{lem:rook}, we get
\begin{align}\label{eq:1}
\EE\left(\sum_{[i,j] \in \sigma} r_{ij}R_{i,j}\right)  &= \sum_{[i,j]\in \sigma} r_{ij}\left( \sum_{[i',j] \in \sigma}\EE(\T^{+}_{i',j} )+ \sum_{[i,j'] \in \sigma} \EE(\T^{+}_{i,j'})\right) \\
&= \sum_{[i,j]\in \sigma} \left( \sum_{[i',j]\in \sigma} r_{i',j}+ \sum_{[i,j']\in \sigma} r_{i,j'}\right) \EE(\T^+_{i,j}) \nonumber \\
&=(a+b) \sum_{[i,j]\in \sigma} \EE(\T^+_{i,j}) \nonumber
\end{align}
On the other hand, by Lemma~\ref{lem:corners},
\begin{align}\label{eq:2}
\EE\left(\sum_{[i,j] \in \sigma} r_{ij}R_{i,j}\right) &=\sum_{[i,j]\in \sigma} r_{ij}\left( 1+ \sum_{c\in C_{ij}(\sigma)}\PP_{\mu}(c) \right) \\
&=\sum_{[i,j]\in \sigma} r_{ij} + \sum_{[i,j]\in \sigma} r_{ij}\sum_{c\in C_{ij}(\sigma)}\PP_{\mu}(c) \nonumber \\
&=ab + \sum_{c\in C(\sigma)}
\left(\sum_{\substack{[i,j]\in\sigma \text{ with} \\C_{ij}(\sigma)\ni c}} r_{ij}\right)\PP_{\mu}(c) \nonumber
\end{align}
As depicted in Figure \ref{fig:cornervalues}, for any corner $c\in C(\sigma)$ occurring at $(x,y)$ and for any~$[i,j]\in \sigma$, we have $~c \in C_{ij}(\sigma)$ if and only if ($x \ge i$ and $y\ge j$) or ($x<i$ and $y<j$). Let $c \in C(\sigma)$ be a southeast corner occurring at $(x,y)$. We have
\begin{align*}
\sum_{\substack{[i,j]\in\sigma  \text{ with}\\ C_{ij}(\sigma)\ni c}} r_{ij}=& \sum_{[i,j]\in \sigma} r_{ij} - \sum_{\substack{[i,j]\in\sigma \\ i\le x ; j> y}} r_{ij} - \sum_{\substack{[i,j]\in\sigma \\ i > x ; j\le y }} r_{ij}\\
&=ab-(a-x)b -(b-y)a\\
&= ab\left(\frac{x}{a}+\frac{y}{b}-1\right).
\end{align*}
With similar calculations we can see for any $c \in C(\sigma)$ a northeast corner occurring at~$(x,y)$ we also have $\sum_{\substack{[i,j]\in\sigma\\ C_{ij}(\sigma)\ni c}} r_{ij}=ab(1-\frac{x}{a}-\frac{y}{b})$. In other words, for $c\in C(\sigma)$,
\begin{equation} \label{eq:3}
\sum_{\substack{[i,j]\in\sigma \\ c\in C_{ij}(\sigma)}} r_{ij}=ab\cdot\delta(c).
\end{equation}
Putting equations (\ref{eq:1}), (\ref{eq:2}) and (\ref{eq:3}) together yields
\[(a+b) \sum_{[i,j]\in\sigma}\EE(\T^+_{i,j}(\sigma))=ab\left(1 + \sum_{c\in C(\sigma)} \delta(c) \PP(c)\right)
.\]
But since $\mu$ is a toggle-symmetric measure, $\EE_{\mu}(\jag)=2\sum_{[i,j]\in \sigma}\EE(\T^+_{i,j}(\sigma))$.  Hence the claimed formula for $\EE_{\mu}(\jag)$ holds.
\end{proof}

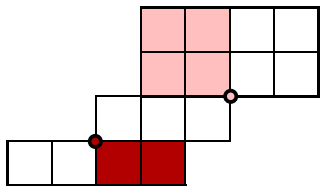
\begin{figure}
\definecolor{DRed}{rgb}{0.7,0,0}
\begin{tikzpicture}
\node at (0,0) {\begin{ytableau} \none & \none & \none & *(pink) &  *(pink)&   &  \\ \none & \none & \none &*(pink)  &*(pink) && \\
\none & \none &  &    & \\ &  & *(DRed) &*(DRed)\\ \end{ytableau}};
\def\x{0.6}
\fill[black] (-1.5*\x,-1*\x) circle (0.1cm);
\fill[DRed] (-1.5*\x,-1*\x) circle (0.05cm);
\fill[black] (1.5*\x,0*\x) circle (0.1cm);
\fill[pink] (1.5*\x,0*\x) circle (0.05cm);

\end{tikzpicture}
\caption{In the above diagram, let~$X_1$ be the set of pink boxes and~$X_2$ the set of dark red boxes. Let $c_1$ be the corner occurring at $(2,5)$ (in pink) and $c_2$ the corner at $(3,2)$ (in dark red). Then $ [i,j]\in X_1$ if and only if~$c_1 \in C_{i,j}(\sigma)$ and~$[i,j]\in X_2$ if and only if~$c_2 \in C_{i,j}(\sigma)$.} \label{fig:cornervalues}
\end{figure}

Let us say a skew shape $\sigma$ is \emph{balanced} if it is connected and $\delta(c)=0$ for all $c \in C(\sigma)$. In other words, a connected skew shape is balanced if all outward corners occur at the main anti-diagonal. An immediate corollary of our main theorem is the following:

\begin{cor} \label{cor:main}
Let $\sigma$ be a balanced skew shape with height $a$ and width $b$.  Let $\mu$ be any toggle-symmetric probability distribution on $\J(\sigma)$.  Then the expected jaggedness of a subshape in $\J(\sigma)$ with respect to the distribution $\mu$ is $\frac{2ab}{a+b}$.
\end{cor}

\begin{figure}
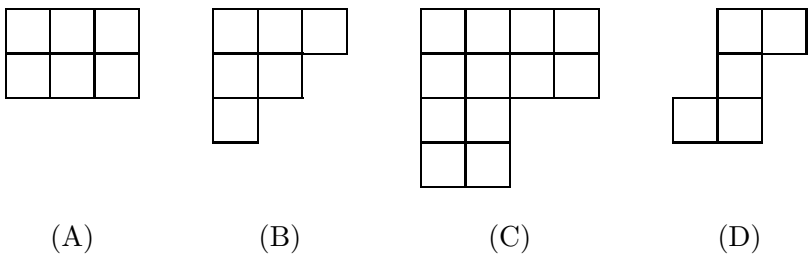

\setlength{\tabcolsep}{0.5cm}
\begin{tabular}{c c c c}
\ydiagram{3,3} & \ydiagram{3,2,1} & \ydiagram{4,4,2,2} & \ydiagram{1+2,1+1,2} \\
\\
(A) & (B) & (C) & (D)
\end{tabular}
\caption{Examples of balanced skew shapes.}\label{fig:balancedshapes}
\end{figure}

Some examples of balanced skew shapes are depicted in Figure~\ref{fig:balancedshapes}. They include rectangles like (A), staircases like (B), ``stretched'' staircases (i.e.,~staircases where we have replaced each box by a~$k \times l$ rectangle) like (C), as well as other more general shapes like (D). There are a total of $3^{\mathrm{gcd}(a,b)-1}$ balanced skew shapes with height $a$ and width~$b$ for any~$a,b \geq 1$. 

Along the same lines, say that a skew shape $\sigma$ is {\em abundant} if all of its northwest corners occur on or above its antidiagonal and all of its southeast corners occur on or below its antidiagonal.  Let us say that $\sigma$ is {\em deficient} if all of its northwest corners occur on or below the antidiagonal and all of its southeast corners occur on or above the antidiagonal.  Then we immediately get:
\begin{cor}\label{cor:abundantdeficient}
Let $\sigma$ be a skew shape of height $a$ and width $b$ and $\mu$ be any toggle-symmetric probability distribution on $\J(\sigma)$.  
\begin{itemize}
\item If $\sigma$ is abundant, then the expected jaggedness of a subshape in $\J(\sigma)$ with respect to the distribution $\mu$ is at least ${2ab}/(a\!+\!b)$.  
\item If $\sigma$ is deficient, then the expected jaggedness of a subshape with respect to $\mu$ is at most ${2ab}/(a\!+\!b)$.
\end{itemize}
\end{cor}

We remark that in the case where $\sigma$ is a rectangle and $\mu = \mu_{\lin}$ or $\mu_{\unif}$, Corollary~\ref{cor:main} recovers a result of Chan et al.~\cite{chan2015genera}. Indeed, Remark 2.16 of~\cite{chan2015genera} points out the ``remarkable'' fact that for rectangles, the uniform and linear distributions have the same expected jaggedness; Corollary~\ref{cor:main} is a vast generalization, and perhaps even explanation, of this phenomenon.  Theorem~\ref{thm:main} also gives a reformulation of \cite[Theorem 2.8]{chan2015genera} (which deals with $\mu_{\lin}$ only) which exhibits more explicitly the way in which expected jaggedness depends on the shape of $\sigma$.

\subsection{Computing the correction terms for various toggle-symmetric distributions} \label{sec:correction}

Although Theorem~\ref{thm:main} gives an especially nice formula for $\EE_{\mu}(\jag)$ when~$\sigma$ is a balanced, even when $\sigma$ is not balanced the correction term~$\sum_{c \in C(\sigma)}\delta(c)\PP_{\mu}(c)$ in this formula is easy to compute for all of the ``natural'' toggle-symmetric distributions defined in Section~\ref{sec:pparts}, as we now explain. Of course the displacement $\delta(c)$ for $c \in C(\sigma)$ is easily computed; the issue is computing $\PP_{\mu}(c)$. By Remarks~\ref{rem:uniform} and~\ref{rem:ranked}, the distributions $\mu_{\unif}$ and $\mu_{\rk}$ are special cases of~$\mu_{m,\leq}$ and~$\mu_{m,<}$, respectively, so from now on we discuss computing~$\PP_{\mu_{m,\leq}}(c)$, $\PP_{\mu_{\lin}}(c)$, and $\PP_{\mu_{m,<}}(c)$. 

First let us consider $\mu = \mu_{m,\leq}$; the other distributions will be similar. Let $\mathrm{RPP}(\sigma;m)$ denote the set of reverse plane partitions of shape $\sigma$ and height $m$ (recalling the dictionary of terms above). Choose some outward corner~$c\in C(\sigma)$ that occurs $(i,j)$. Suppose first that $c$ is a southeast corner. Then
\[\PP_{\mu_{m,\leq}}(c) = \EE\left(\frac{m-\mathrm{max}\{\ell([i+1,j]),\ell([i,j+1])\}}{m}\right)\]
where $\ell \in \mathrm{RPP}(\sigma;m)$ is chosen uniformly at random. But
\[\EE(\mathrm{max}\{\ell([i+1,j]),\ell([i,j+1])\}) = m+1 - \frac{\#\mathrm{RPP}(\sigma\cup\{c\};m)}{\#\mathrm{RPP}(\sigma;m)}\]
where $\sigma \cup \{c\}$ denotes the skew shape obtained by adding a box at corner $c$ (i.e., for this southeast corner, $\sigma \cup \{c\} := \sigma \cup \{[i+1,j+1]\}$). Indeed, this follows from the same observation as Lemma 2.10 of~\cite{chan2015genera}: consider the map $\mathrm{RPP}(\sigma\cup\{c\};m) \to \mathrm{RPP}(\sigma;m)$ given by forgetting the value at $[i+1,j+1]$; for any $\ell \in \mathrm{RPP}(\sigma;m)$ the size of the fiber of this map at $\ell$ is $m+1 - \mathrm{max}\{\ell([i+1,j]),\ell([i,j+1])\}$. Now suppose that~$c$ is a northwest corner. Then
\[\PP_{\mu_{m,\leq}}(c) = \EE\left(\frac{\mathrm{min}\{\ell([i+1,j]),\ell([i,j+1])\}}{m}\right)\]
where $\ell \in \mathrm{RPP}(\sigma;m)$ is chosen uniformly at random. By the same reasoning as before,
\[\EE(\mathrm{min}\{\ell([i+1,j]),\ell([i,j+1])\}) = \frac{\#\mathrm{RPP}(\sigma\cup\{c\};m)}{\#\mathrm{RPP}(\sigma;m)} - 1.\]
Whether $c$ is a southeast or northwest corner, we see that
\begin{equation}
\PP_{\mu_{m,\leq}}(c) = \frac{\#\mathrm{RPP}(\sigma\cup\{c\};m)-\#\mathrm{RPP}(\sigma;m)}{m\cdot\#\mathrm{RPP}(\sigma;m)}.
\end{equation}
Similar analysis for the other distributions shows
\begin{equation}\label{eq:mulin}
\PP_{\mu_{\lin}}(c) = \frac{\#\mathrm{SYT}(\sigma\cup\{c\})}{(|\sigma|+1)\cdot \#\mathrm{SYT}(\sigma)}
\end{equation}
where $\mathrm{SYT}(\sigma)$ denotes the set of standard Young tableaux of shape $\sigma$ and $|\sigma|$ is the number of boxes in $\sigma$, and
\begin{equation}\PP_{\mu_{m,<}}(c) = \frac{\#\mathrm{Inc}(\sigma\cup\{c\};m)+\#\mathrm{Inc}(\sigma;m)}{(m+2)\cdot\#\mathrm{Inc}(\sigma;m)}.
\end{equation}
where $\mathrm{Inc}(\sigma;m)$ is the set of increasing tableaux of shape $\sigma$ and height $m$.

Thus we have reduced the problem of computing $\PP_{\mu}(c)$ for $\mu$ in the spectrum of toggle-symmetric distributions on $\J(\sigma)$ defined in Section~\ref{sec:pparts} to computing the quantities~$\#\mathrm{RPP}(\sigma;m)$, $\#\mathrm{SYT}(\sigma)$, and $\#\mathrm{Inc}(\sigma;m)$. Fortunately there are determinantal formulas for these. Let $\sigma = \lambda / \nu$ be a skew shape with $\lambda = (\lambda_1,\ldots,\lambda_k)$ and~$\nu = (\nu_1,\ldots,\nu_k)$. Then a result of Kreweras~\cite{kreweras1965sur} says that
\[\#\mathrm{RPP}(\sigma;m) = \det_{i,j=1}^{k}\left[\binom{\lambda_i - \nu_j + m}{i-j+m}\right]. \]
Here we interpret $\binom{x}{y} = 0$ for $x < 0$. In the special case $\mu = \emptyset$ the above formula was known to MacMahon~\cite[p. 243]{macmahon2004combinatory}; for more details see~\cite[Exercise 3.149]{stanley2012ec1}. The following formula is due to Aitken~\cite{aitken1943monomial} (although in fact it is a simple consequence of the Jacobi-Trudi identity; see~\cite[Corollary 7.16.3]{stanley1999ec2}):
\[ \#\mathrm{SYT}(\sigma) = |\sigma|! \det_{i,j=1}^{k} \left[\frac{1}{(\lambda_i - i -\nu_j +j)!}\right] \]
Here we interpret $\frac{1}{x!} = 0$ if $x < 0$. In the special case $\mu=\emptyset$ we also have the famous Hook-Length Formula, which gives an even better answer for the number of standard Young tableaux; namely,
\[\#\mathrm{SYT}(\lambda) = |\lambda|! \prod_{[i,j] \in \lambda} \frac{1}{h_{\lambda}(i,j)}\]
where $h_{\lambda}(i,j)$ is the hook-length of box $[i,j]$; see~\cite[Corollary 7.21.6]{stanley1999ec2} for details. As for increasing tableaux of bounded height, it follows from the Reciprocity Theorem for order polynomials (see~\cite[Corollary 3.15.12]{stanley2012ec1}) that if $\Omega$ is the unique polynomial satisfying~$\Omega(m) = \#\mathrm{RPP}(\sigma;m)$ for all $m \in \NN$ then~$\#\mathrm{Inc}(\sigma;m) = (-1)^{|\sigma|}\Omega(-m)$. Thus the aforementioned result of Kreweras also allows us to easily compute~$\#\mathrm{Inc}(\sigma;m)$.

\subsection{Connections to antichain cardinality homomesy} \label{ssec:antichain}

In this subsection we give an application of our main result to the study of homomesies in combinatorial maps. Recall the definitions of rowmotion, gyration, and homomesy from Section~\ref{sec:togglegpsym}.

Let $P$ be a poset. To any~$\mcO \in \J(P)$ we associate the antichain $A(\mcO)$ of $P$ consisting of the maximal elements of~$\mcO$. The \emph{antichain cardinality statistic} is the map~$\J(P) \to \RR$ given by~$\mcO \mapsto \#A(I)$.

\begin{cor} \label{cor:homomesy-general}
If $P$ is the poset associated to the skew shape $\sigma$ and $\mu$ is any toggle-symmetric distribution, then
$$\EE_{\mu}(\#A(I)) = \frac{ab}{a+b}\left( 1 + \sum_{c\in C(\sigma)} \!\delta(c)\,\PP_\mu(c)\right).$$
\end{cor}

\begin{proof}
This result was already obtained in the proof of Theorem~\ref{thm:main}.  Explicitly,
the antichain cardinality statistic is just $\sum_{p \in P_{\sigma}}\T^{-}_p$, so the average of this statistic is
\begin{align*}
\EE_{\mu}(\sum_{p \in P_{\sigma}}\T^{-}_p) &= \frac{1}{2}(\EE_{\mu}(\sum_{p \in P_{\sigma}}\T^{-}_p) + \EE_{\mu}(\sum_{p \in P_{\sigma}}\T^{+}_p)) \\
&= \frac{1}{2}\EE_{\mu}(\jag) \\
\end{align*}
where $\EE_{\mu}(\sum_{p \in P_{\sigma}}\T^{-}_p) = \EE_{\mu}(\sum_{p \in P_{\sigma}}\T^{+}_p)$ thanks to the toggle-symmetry of $\mu$.  Now apply Theorem~\ref{thm:main}.
\end{proof}

\begin{cor} \label{cor:homomesy}
For $P_{\sigma}$ the poset corresponding to a balanced skew shape $\sigma$ of height~$a$ and width~$b$ and $\varphi\in\{\mathrm{rowmotion,gyration}\}$, the antichain cardinality statistic is~$\frac{ab}{a+b}$-mesic with respect to the action of $\varphi$ on $\J(P_{\sigma})$.
\end{cor}
\begin{proof}
Let $\Orb \subseteq \J(P_{\sigma})$ be a $\varphi$-orbit and let $\mu$ be the distribution on $\J(P_{\sigma})$ that is uniform on $\Orb$. By Theorem~\ref{thm:striker} we know that $\mu$ is toggle-symmetric. Thus by Corollaries~\ref{cor:main} and~\ref{cor:homomesy-general} we conclude that $\EE_{\mu}(\#A(I)) = \frac{ab}{a+b}$.
\end{proof}

In the case where $\varphi = \mathrm{rowmotion}$ and $\sigma$ is an $a \times b$ rectangle, Corollary~\ref{cor:homomesy} recovers a result of Propp and Roby~\cite[Theorem 27]{propp2013homomesy}. Actually, Propp and Roby prove a more refined result: they show the cardinality of the intersection of the antichain with any fixed ``fiber'' of~$P_{\sigma}$ is homomesic with respect to rowmotion in this rectangular case. In other words, they show that the statistics~$\sum_{[i,j']\in\sigma}T^{-}_{i,j'}$ for~$1 \leq i \leq a$ and~$\sum_{[i',j]\in\sigma}T^{-}_{i',j}$ for~$1 \leq j \leq b$ are homomesic with respect to the action of rowmotion on~$\J(P_{\sigma})$. But when $\sigma$ is an $a \times b$ rectangle and $1 \leq i < a$ we have
\[ \sum_{[i,j']\in\sigma}T^{-}_{i,j'} = \sum_{[(i+1),j']\in\sigma}T^{+}_{(i+1),j'} \]
and similarly for columns. Thus by the toggle-symmetry of $\mu$, where $\mu$ is as in the proof of Corollary~\ref{cor:homomesy}, we conclude that in this case
\[ \EE_{\mu}\left(\sum_{[i_1,j']\in\sigma}T^{-}_{i_1,j'}\right) = \EE_{\mu}\left(\sum_{[i_2,j']\in\sigma}T^{-}_{i_2,j'}\right) \]
for any $1 \leq i_1,i_2 \leq a$, and similarly for columns. In this way we can recover Propp and Roby's refined fiber cardinality result as well. This argument also shows that fiber cardinality is homomesic for gyration acting on rectangular shapes. (But note that the fiber cardinality homomesy does not hold for arbitrary balanced shapes.) At any rate, for non-rectangular, balanced $\sigma$ when $\varphi = \mathrm{rowmotion}$, and for all balanced~$\sigma$ when~$\varphi=\mathrm{gyration}$, the antichain cardinality homomesy result of Corollary~\ref{cor:homomesy} appears to be new.

\section{Open questions}\label{sec:open}

We conclude with some open questions and possible threads of future research.

\begin{enumerate}

\item For any poset $P$, the space of toggle-symmetric distributions on $\J(P)$ is some convex polytope. Denote this polytope by $\Poly(P)$. What is the combinatorial structure of~$\Poly(P)$? Note that $\Poly(P)$ has dimension $\#\J(P) - 1 - \#P$: specifically, it is the intersection of the standard $\#\J(P)$-simplex in $\RR^{\#\J(P)}$ with some linear subspace of codimension~$\#P$, and the uniform distribution on~$\J(P)$ is an interior point of the simplex that is always toggle-symmetric. It seems that $\Poly(P)$ can be rather complicated; for example, computation with Sage mathematical software shows that when~$P_{\lambda}$ is the poset corresponding to the partition~$\lambda = (3,3,3)$ the polytope~$\Poly(P_{\lambda})$ is~$10$-dimensional and has~$159$ vertices. For a specific question about~$\Poly(P)$: are the distributions corresponding to $\varphi$-orbits for~$\varphi\in\{\mathrm{rowmotion,gyration}\}$ always vertices of $\Poly(P)$?

\item Rowmotion and gyration are both elements of the toggle group; moreover, they are both compositions of all of the toggles in some order. Not all such compositions of toggles are $0$-mesic with respect to $\T_p$ for all $p \in P$; for instance, Striker~\cite[\S6]{striker2015toggle} observes an instance where this fails for \emph{promotion}, another such element of the toggle group. Nevertheless, we could hope that there were some other toggle group elements $\varphi\colon \J(P) \to \J(P)$ that are $0$-mesic with respect to $\T_p$ for all~$p \in P$. It would be interesting to find such $\varphi$ because then Corollary~\ref{cor:main} would immediately imply that the antichain cardinality statistic is homomesic with respect to~$\varphi$.

\item For a connected skew shape $\sigma$ with height $a$ and width $b$ and any $\rho \in \J(\sigma)$ we claim that
\[ 1 \leq \jag(\rho) \leq \mathrm{min}\{2a,2b,a+b-1\}<\frac{4ab}{a+b}.\]
To see this, first note that either $\rho$ is nonempty or $\sigma \setminus \rho$ is nonempty and so there is at least one box of $\sigma$ that can be toggled in or out, proving~$1 \leq \jag(\rho)$. Next note that in each column, at most one box can be toggled in and at most one out, and similarly for rows. This proves $ \jag(\rho) \leq \mathrm{min}\{2a,2b\}$. The only case where $a+b-1 <  \mathrm{min}\{2a,2b\}$ is when $a=b$; in this case, note that if a box can be toggled out of every column, then there is no box in the first column that can be toggled in. So indeed the claimed inequality on $\jag(\rho)$ holds. The upshot of this inequality is that for any distribution $\mu$ on~$\J(\sigma)$,
\[0 < \EE_{\mu}(\jag) < \frac{4ab}{a+b}. \]
If $\mu$ is toggle-symmetric then by Theorem~\ref{thm:main} we conclude
\[-1 < \sum_{c \in C(\sigma)}\delta(c)\PP_{\mu}(c) < 1.\]
It is not obvious a priori that this bound on $\sum_{c \in C(\sigma)}\delta(c)\PP_{\mu}(c)$ should hold for all toggle-symmetric distributions $\mu$. It would be interesting to give a simple explanation for why it does hold, or to offer another expression for $\EE_{\mu}(\jag)$ that is evidently strictly between $0$ and $\frac{4ab}{a+b}$.

A related question, pointed out by N.~Pflueger, is to give a direct explanation for why, for any {\em balanced} skew shape $\sigma$ and any toggle-symmetric distribution, the expected jaggedness of a subshape necessarily lies between $a$ and $b$.  (This is true, of course, since the harmonic mean always lies between $a$ and $b$.)

\item Our main theorem, Theorem~\ref{thm:main}, which gives a formula for $\EE_{\mu}(\jag)$ for toggle-symmetric distributions $\mu$, applies only to posets associated to skew shapes. Can we generalize this result to a broader class of posets? In particular, is there a more general notion of a ``balanced'' poset for which all toggle-symmetric distributions have the same expected jaggedness?

\end{enumerate}

\bibliography{main}{}
\bibliographystyle{plain}

\end{document}